\documentclass[12pt]{amsart}
\usepackage{amssymb}
\usepackage{epsfig}
\usepackage{times}
\numberwithin{equation}{section}

\input xy
\xyoption{all}

\makeatother
\makeatletter

\def\Title    {Reconstruction of function fields}
\def\Author   {Fedor Bogomolov and Yuri Tschinkel}
\def\Subject  {Algebraic geometry}
\def\Keywords {Galois groups, function fields}

\usepackage[pdftex,
hyperindex=true,
pdftitle={\Title},pdfauthor={\Author},pdfsubject={\Subject},
pdfkeywords={\Keywords},pdfpagemode={UseOutlines},
bookmarks=true,bookmarksopen=true,
bookmarksopenlevel=0,bookmarksnumbered=true,
pdfstartview={FitV},
colorlinks=true
]{hyperref}

\usepackage{graphicx}

\advance\headheight 2pt
\calclayout
\allowdisplaybreaks[3]

\theoremstyle{plain}
\newtheorem{prop}[subsection]{Proposition}
\newtheorem{Theo}[section]{Theorem}
\newtheorem{thm}[subsection]{Theorem}
\newtheorem{coro}[subsection]{Corollary}

\newtheorem{lemm}[subsection]{Lemma}

\newtheorem{defn}[subsection]{Definition}

\theoremstyle{definition}

\theoremstyle{remark}

\newtheorem{exam}[subsection]{Example}

\newtheorem{nota}[subsection]{Notation}

\def\G{{\mathcal G}}   
\def\D{{\mathcal D}}   
\def\I{{\mathcal I}}   

\def\pic{{\varphi}}
\def\Ga{{\Gamma}}
\def\Val{{\mathcal V}}
\def\DVal{{\mathcal D}{\mathcal V}}

\def\KK{{\boldsymbol{K}}}
\def\LL{{\boldsymbol{L}}}

\def\Alb{{\rm Alb}}

\def\supp{{\rm supp}}

\def\res{{\rm res}}

\def\no{\noindent}
\def\rk{{\rm rk}}
\newcommand{\trdeg}{{\rm tr}\, {\rm deg}}
\newcommand{\ovl}{\overline}

\def\cH{{\mathcal H}}

\def\cT{{\mathcal T}}

\def\dv{{\rm div}}
\newcommand{\Hom}{{\rm Hom}}
\newcommand{\Ker}{{\rm Ker}}

\newcommand{\fibr}{{\rm fibr}}

\def\lra{\longrightarrow}
\def\ra{\rightarrow}

\def\F{{\mathbb F}}

\def\P{{\mathbb P}}
\def\Q{{\mathbb Q}}

\def\Z{{\mathbb Z}}

\def\N{{\mathbb N}}

\def\sc{{\mathsf c}}

\def\mo{{\mathfrak o}}

\def\mm{{\mathfrak m}}

\def\Pic{{\rm Pic}}
\def\Div{{\rm Div}}

\def\pr{{\rm pr}}

\author{Fedor Bogomolov}
\address{Courant Institute of Mathematical Sciences, N.Y.U. \\
 251 Mercer str. \\
 New York, NY 10012, U.S.A.}
\email{bogomolo@cims.nyu.edu}

\author{Yuri Tschinkel}
\address{Courant Institute of Mathematical Sciences, N.Y.U. \\
 251 Mercer str. \\
 New York, NY 10012, U.S.A.}
\email{tschinkel@cims.nyu.edu}        

\title[Function fields]{Reconstruction of higher-dimensional 
function fields}

\begin{document}

\begin{abstract}
We determine the function fields of varieties of dimension $\ge 2$ 
defined over the algebraic closure of $\mathbb{F}_p$, modulo purely 
inseparable extensions, from the quotient by the second term 
in the lower central series of their pro-$\ell$ Galois groups. 
\end{abstract}
\date{\today}

\maketitle
\tableofcontents

\setcounter{section}{0}       
\section*{Introduction}
\label{sect:introduction}

Fix two distinct primes $p$ and $\ell$. 
Let $k=\ovl{\F}_p$ be an algebraic closure of
the finite field $\F_p$. 
Let $X$ be an algebraic variety defined
over $k$ and $K=k(X)$ its function field.
We will refer to $X$ as a {\em model} of $K$;
we will generally assume that $X$ is normal and projective.  
Let $\G^a_K$ be the abelianization of 
the pro-$\ell$-quotient $\G_K$ of the
absolute Galois group of $K$. 
Under our assumptions on $k$, $\G^a_K$ 
is a torsion-free $\Z_{\ell}$-module.  
Let  $\G^c_K$ be its canonical
central extension - the second lower 
central series quotient of $\G_K$.
It determines a set $\Sigma_K$ of
distinguished (primitive) finite-rank subgroups: 
a topologically noncyclic subgroup $\sigma\in \Sigma_K$ iff
\begin{itemize}
\item $\sigma$ lifts to an abelian subgroup of $\G^c_K$;
\item $\sigma$ is maximal:
there are no abelian subgroups $\sigma'\subset \G^a_K$
which lift to an abelian subgroup of $\G^c_K$
and contain $\sigma$ as a proper subgroup.
\end{itemize}
Our main theorem is

\begin{Theo}                                     
\label{thm:main}
Let $K$ and $L$ be function fields over
algebraic closures of finite fields $k$, resp. $l$,  
of characteristic $\neq \ell$. 
Assume that the transcendence degree of $K$ over $k$ is at least two and  
that there exists an isomorphism
\begin{equation}
\label{eqn:psi-dual}
\Psi=\Psi_{K,L}\,:\, \G^a_K\stackrel{\sim}{\lra} \G^a_{L}
\end{equation}
of abelian pro-$\ell$-groups inducing a bijection of sets
$$
\Sigma_K = \Sigma_{L}.
$$
Then $k=l$ and there exists a constant $\epsilon\in \Z_{\ell}^*$
such that $\epsilon^{-1}\cdot\Psi$ is induced from 
a unique isomorphism of perfect closures
$$
\bar{\Psi}^*\,:\, \bar{L}\stackrel{\sim}{\lra} \bar{K}.
$$
\end{Theo}

\no

In this paper we implement the program outlined in 
\cite{B-1} and \cite{B-3} describing the correspondence between 
higher-dimensional function fields and their abelianized Galois
groups. We follow closely our paper \cite{bt0}, 
where we treated in detail the case
of surfaces:
The isomorphism \eqref{eqn:psi-dual} 
of abelianized Galois groups induces a canonical 
isomorphism
$$
\Psi^*\,:\, \hat{L}^*\stackrel{\sim}{\lra} \hat{K}^*
$$
between pro-$\ell$-completions of multiplicative groups.
One of the steps in the proof is to show 
that under the assumptions of Theorem~\ref{thm:main}, $\Psi^*$ induces
by restriction a canonical isomorphism
\begin{equation}
\label{eqn:induct}
\Psi^*\,:\, L^*/l^*\otimes \Z_{(\ell)} \stackrel{\sim}{\lra} 
\left(K^*/k^*\otimes \Z_{(\ell)}\right)^\epsilon, \quad \text{ for some } \quad \epsilon \in \Z_{\ell}^*.
\end{equation}
Here we proceed by induction 
on the transcendence degree, using \cite{bt0}
as the inductive assumption. 
We first recover abelianized inertia and decomposition 
subgroups of divisorial valuations 
using the theory of commuting pairs developed in \cite{BT}. 
Then we apply the inductive assumption~\eqref{eqn:induct} to 
residue fields of divisorial valuations. This allows to prove
that for every normally closed one-dimensional subfield $F=l(f)\subset L$
there exists a one-dimensional subfield $E\subset K$ such that 
$$
\Psi^*(F^*/l^* \otimes \Z_{(\ell)}) \subseteq \left(E^*/k^*\otimes \Z_{(\ell)}\right)^\epsilon,
$$
for some constant $\epsilon\in \Z_{\ell}^*$, depending on $F$. The proof that $\epsilon$ 
is independent of $F$ 
and, finally, the proof of Theorem~\ref{thm:main}
are then identical to those in dimension two in \cite{bt0}.

\

\no
{\bf Acknowledgments.} 
The first author was partially supported  by NSF grant DMS-0701578. 
The second author was partially supported by NSF grants DMS-0739380 and 
0901777.

\section{Basic algebra and geometry of fields}
\label{sect:basicalg}

Here we state some auxiliary facts used in the proof of our main theorem.

\begin{lemm}
\label{lemm:normal}
Every function field over an algebraically closed ground field admits a normal model.
\end{lemm}

\begin{lemm}
\label{lemm:subf}
For every one-dimensional subfield $E\subset K$ there is a canonical
sequence of maps 
$$
X\stackrel{\pi_E}{\lra} C'\stackrel{\mu_E}{\lra} C,
$$
where 
\begin{itemize}
\item $\pi_E$ is dominant with irreducible generic fiber;
\item $\mu_E$ is quasi-finite and dominant; 
\item $k(C')=\ovl{E}^{K}$ (the normal closure of $E$ in $K$) and $k(C)=E$.
\end{itemize}
\end{lemm}
Note that $C'$ and $C$ do not depend on the choice of a model $X$. 

\

We call a divisor $p$-irreducible if it is a $p$-power of an irreducible divisor.

\begin{lemm} 
\label{lemm:ea1}
Let $\pi\,:\, X\ra C$ be a surjective map with irreducible generic fiber
and $R\subset X$ an irreducible divisor surjecting onto $C$. 
Then the intersection $R\cdot \pi^{-1}(c)$ is $p$-irreducible for 
all but finitely many $c\in C$. 
\end{lemm}

\begin{proof}
This is a positive-characteristic version of Bertini's theorem. 
\end{proof}

\begin{lemm}
\label{lemm:hurww}
Let $\pi \,:\, T\ra C$ be a separable map of degree $m$ with branch locus 
$\{c_1,\ldots, c_N\}\subset C$. 
Write
$$
\pi^{-1}(c_j) = \sum_{r=1}^{m_j} e_{j,r} t_{j,r}, \quad 
t_{j,r} \in T, e_{j,r}\in \N, 
\,\,\text{ and } \quad  \sum_{r=1}^{m_j} e_{j,r} = m.
$$
Let $e'_{j,r}$ be the maximal prime-to-$p$ divisor of $e_{j,r}$.
Assume that
$$
\sum_{r=1}^{m_j} (e'_{j,r}-1) > m/2,
$$
for all $j=1,\ldots, N$.
Then
$$
\mathsf g(T) > N-3.
$$
\end{lemm}

\begin{proof}
Hurwitz formula (for curves over a field
of finite characteristics). 
\end{proof}

Let $X\subset \P^N$ be a normal projective variety of dimension $n\ge 2$ over $k$. 
Consider the moduli space $\mathcal M(d)$ 
of complete intersection curves on $X$ of multidegree $d=(d_1,\ldots, d_{n-1})$. 
For $|d|\gg 0$ we have:
\begin{itemize}
\item 
for any codimension $\ge 2$ subvariety $Z\subset X$
there is a Zariski open subset 
of $\mathcal M(d)$ such that every curve $C$ 
parametrized by a point in this subset
avoids $Z$ and intersects every 
irreducible divisor $D\subset X$.
\end{itemize} 
Such families will be called {\em families of flexible curves}.

\

A {\em Lefschetz pencil} 
is a surjective map
$$
\lambda\,:\, X\ra \P^1
$$
from a normal variety with irreducible fibers and normal generic fiber. 

\begin{lemm}
\label{lemm:genus-generic}
Let $\lambda\,:\, X\ra \P^1$ be a Lefschetz pencil on a normal projective variety. 
Then there exists an $m\in \N$
such that every irreducible normal fiber $D_t:=\lambda^{-1}(t)$ contains
a family of flexible curves of genus $\le m$.
\end{lemm}

\begin{proof}
We can realize the fibers $D_t$
simultaneously as hyperplane sections in a fixed projective embedding
and consider induced complete intersection curves. 
The degree calculation for $X$ yields a uniform genus estimate for
corresponding flexible curves for all $D_t$.  
\end{proof}

The following lemma is a consequence of the Merkuriev--Suslin theorem:

\begin{lemm} \cite[Lemma 25]{bt-milnor}.
\label{lemm:milnor}
Let $K=k(X)$ be a function field of an algebraic variety of dimension $\ge 2$.
Two functions $f_1,f_2\in K^*/k^*$ are algebraically dependent if and only if
the symbol in the (reduced) second Milnor {\rm K}-group of $K$ vanishes:
$$
(f_1,f_2) = 0\in {\rm K}^M_2(K)/ \text{ infinitely divisible elements}.
$$
\end{lemm}

\section{Galois groups}
\label{sect:gal}

Let $\G^a_K$ the abelianization of the
pro-$\ell$-quotient $\G_K$ of the Galois group  of a separable
closure of $K=k(X)$, 
$$
\G^c_K=\G_K/[[\G_K,\G_K],\G_K]\stackrel{\pr}{\lra} \G^a_K
$$
its canonical central extension and $\pr$ the natural projection.
By our assumptions, $\G^a_K$ is a torsion-free $\Z_{\ell}$-module.

\begin{defn}
\label{defn:lift}
We say that $\gamma,\gamma'\in \G^a_K$ form a commuting pair
if for some (and therefore any) 
of their preimages 
$\tilde{\gamma}\in \pr^{-1}(\gamma),\tilde{\gamma}' \in \pr^{-1}(\gamma') \in \G_K^c$, one has
$
[\tilde{\gamma},\tilde{\gamma}']=0.
$
A subgroup $\cH$ of $\G^a_K$ is called liftable
if any two elements in $\cH$ form a commuting pair. 
A liftable subgroup is called maximal if it is not properly 
contained in any other liftable subgroup. 
\end{defn}

\begin{defn} 
\label{defn:fan}
A fan $\Sigma_K=\{ \sigma\}$ on $\G^a_K$ is 
the set of all topologically noncyclic maximal
liftable subgroups $\sigma\subset\G^a_K$.
\end{defn}

\begin{nota}
\label{nota:kkk}
Let 
$$
\boldsymbol{\mu}_{\ell^n}:=\{ \sqrt[\ell^n]{1}\,\}
$$
and 
$$
\Z_{\ell}(1) =\lim_{n\ra \infty} \boldsymbol{\mu}_{\ell^n}
$$
be the Tate twist of $\Z_{\ell}$.
Write 
$$
\hat{K}^*:=\lim_{n\ra \infty} K^*/ (K^*)^{\ell^n}
$$
for the $\ell$-adic completion of the multiplicative group $K^*$. 
\end{nota}

\begin{thm}[Kummer theory]
\label{thm:ga} 
For every $n\in \N$ we have a pairing 
$$
\begin{array}{ccl}
[ \cdot , \cdot ]_n \,:\, \G^a_K/\ell^n\times K^*/(K^*)^{\ell^n} & \ra & \boldsymbol{\mu}_{\ell^n} \\
   \,\,\,\,\,\, (\mu,f)                          & \mapsto & [\mu,f]_n:= \mu(f)/f
\end{array}
$$
which extends to a nondegenerate pairing 
$$
[ \cdot ,\cdot ]\,:\, 
\G^a_K\times \hat{K}^*\ra \Z_{\ell}(1).
$$
\end{thm}

Since $k$ is algebraically closed of characteristic $\neq \ell$ we can 
choose a noncanonical isomorphism 
$$
\Z_\ell\simeq \Z_{\ell}(1).
$$ 
From now on we will fix such a choice.

\section{Valuations}
\label{sect:val}

In this section we recall basic definitions and facts concerning valuations, 
and their inertia and decomposition subgroups of Galois groups (see \cite{bourb}).

\

A (nonarchimedean) {\em valuation} $\nu=(\nu,\Ga_{\nu})$ on $K$
is a pair consisting of a  totally ordered 
abelian group $\Ga_{\nu}=(\Gamma_\nu,+)$ (the {value} group) 
and a map
$$
\nu\,:\, K\ra \Ga_{\nu,\infty}:=\Gamma_\nu \cup \{ \infty\} 
$$
such that
\begin{itemize}
\item $\nu\,:\, K^*\ra \Ga_{\nu}$ is a surjective homomorphism;
\item $\nu(\kappa+\kappa')\ge 
\min(\nu(\kappa),\nu(\kappa'))$ for all $\kappa,\kappa'\in K$;
\item $\nu(0)=\infty$.
\end{itemize}
Every valuation of $K=k(X)$ restricts to a trivial valuation on $k=\overline{\F}_p$.

Let $\mo_{\nu}, \mm_{\nu}$ and $\KK_{\nu}$ 
be the ring of $\nu$-integers in $K$,
the maximal ideal of $\mo_{\nu}$ and the residue field
$$
\KK_{\nu}:=\mo_{\nu}/\mm_{\nu}.
$$
Basic invariants of valuations are:
the $\Q$-rank $\rk_{\Q}(\Gamma_{\nu})$ of the value group $\Gamma_{\nu}$ 
and the transcendence degree
$\trdeg_k(\KK_{\nu})$ of the residue field. We have:
\begin{equation}
\label{eqn:rank}
\rk_{\Q}(\Gamma_{\nu}) + \trdeg_k(\KK_{\nu}) \le \trdeg_k(K).
\end{equation}
A valuation on $K$ has an algebraic center 
$\mathfrak c_{\nu,X}$ on every model $X$ of $K$;
there exists a model $X$ 
where the dimension of $\mathfrak c_{\nu,X}$ is maximal, 
and equal to $\trdeg_{k}(\KK_{\nu})$.
A valuation $\nu$ is called {\em divisorial} if
$$
\trdeg_{k}(\KK_{\nu})=\dim(X)-1;
$$
it can be realized as the discrete rank-one valuation
arising from a divisor on some normal model $X$ of $K$.  
We let $\Val_K$ be the set 
of all nontrivial (nonarchimedean) valuations of $K$ and 
$\DVal_K$ the subset of its divisorial valuations.

It is useful to keep in mind the following exact sequences:
\begin{equation}
\label{eqn:1}
1\ra \mo_{\nu}^*\ra K^*\ra \Ga_{\nu}\ra 1
\end{equation}
and
\begin{equation}
\label{eqn:2}
1\ra (1+\mm_{\nu})^*\ra \mo_{\nu}^*\ra \KK_{\nu}^*\ra 1.
\end{equation}

For every $\nu\in \Val_K$ we have the diagram
$$
\begin{array}{ccccc}
\I^c_\nu & \subseteq & \D^c_{\nu}  & \subset & \G^c_K \\
\downarrow &       &  \downarrow &         & \downarrow \\
\I^a_\nu & \subseteq & \D^a_{\nu}  & \subset & \G^a_K,  
\end{array}
$$
where $\I^c_{\nu},\I^a_{\nu}, \D^c_{\nu}, \D^a_{\nu}$ are the 
images of the inertia and the decomposition group of
the valuation $\nu$ in $\G^c_K$, respectively, $\G^a_K$; 
the left arrow is an isomorphism and the other arrows surjections. 
There are canonical isomorphisms 
$$
\D^c_{\nu}/\I^c_{\nu}\simeq \G^c_{\KK_{\nu}} \quad \text{ and }  \quad \D^a_{\nu}/\I^a_{\nu}\simeq \G^a_{\KK_{\nu}}.
$$
The group $\D^c_{\nu}$ is the centralizer of $\I^c_{\nu}=\I^a_{\nu}$ in $\G^c_{\nu}$, i.e., 
$\I^a_{\nu}$ is the subgroup of elements forming a commuting pair with every element of $\D^a_{\nu}$.

For divisorial valuations $\nu\in \DVal_K$, we have
\begin{equation}
\label{eqn:div}
\I^c_{\nu}= \I^a_{\nu}\simeq \Z_{\ell}.
\end{equation}
Kummer theory, combined with equations \eqref{eqn:1} and \eqref{eqn:2} yields
\begin{equation}
\label{eqn:iav}
\I^a_{\nu}= \{ \gamma\in \Hom(K^*,\Z_{\ell})\,|\, \gamma \,\text{ trivial on } \,\mo_{\nu}^*\}=
\Hom(\Gamma_{\nu}, \Z_{\ell})
\end{equation}
and 
\begin{equation}
\label{eqn:dav}
\D^a_{\nu} = \{\gamma\in  \Hom(K^*,\Z_{\ell}) \,|\, \gamma  \,\text{ trivial on } \, (1+\mm_{\nu})^*\}.
\end{equation}
In particular, 
\begin{equation}
\label{eqn:rank-nu}
\rk_{\Z_{\ell}}(\I^a_{\nu})\le \rk_{\Q}(\Gamma_{\nu}) \le \trdeg_k(K).
\end{equation}
Two valuations $\nu_1,\nu_2$ are dependent
if there exists a common coarsening valuation $\nu$ 
(i.e., $\mm_{\nu}$ is contained in both $\mm_{\nu_1}, \mm_{\nu_2}$), 
in which case
$$
\D^a_{\nu_1}, \D^a_{\nu_2}\subset \D^a_{\nu}. 
$$ 
For independent valuations $\nu_1,\nu_2$ we have
$$
K^*=(1+\mm_{\nu_1})^*(1+\mm_{\nu_2})^*;
$$
it follows that their decomposition groups
have trivial intersection. 

In \cite[Proposition 6.4.1, Lemma 6.4.3 and Corollary 6.4.4]{BT} we proved:

\begin{prop}
\label{prop:lift}
Every topologically noncyclic liftable subgroup $\sigma$ of 
$\G^a_K$ contains a subgroup $\sigma'\subseteq \sigma$
such that there exists a valuation $\nu\in \Val_K$ with 
$$
\sigma'\subseteq \I^a_{\nu},  \quad \sigma\subseteq \D^a_{\nu},
$$
and $\sigma/\sigma'$ topologically cyclic. 
\end{prop}

\begin{coro}
\label{coro:sigma-rank}
For every $\sigma\in \Sigma_K$ one has
$$
\rk_{\Z_{\ell}}(\sigma)\le \trdeg_k(K).
$$
\end{coro}

\begin{proof}
By \eqref{eqn:rank-nu}, 
$$
\rk_{\Z_{\ell}}(\I^a_{\nu})\le \trdeg_k(K).
$$
We are done if $\sigma=\sigma'$. Otherwise, $\D^a_{\nu}/\I^a_{\nu}$ is nontrivial 
and $\trdeg_k(\KK_{\nu})\ge 1$. In this case, \eqref{eqn:rank-nu} and \eqref{eqn:rank} yield that 
$$
\rk_{\Z_{\ell}}(\sigma')\le \trdeg_k(K)-1,
$$ 
and the claim follows.
\end{proof}

\begin{coro}
\label{coro:intersect}
Assume that for $\sigma_1,\sigma_2\in \Sigma_K$ one has 
$$
\sigma_1\cap \sigma_2\neq 0. 
$$
Then there exists a valuation $\nu\in \Val_K$ such that 
$$
\sigma_1, \sigma_2\subset \D^a_{\nu}.
$$
\end{coro}

\begin{proof}
The valuations cannot be independent. Thus there exists a common coarsening.
\end{proof}

This allows to recover the abelianized decomposition and inertia groups of valuations 
in terms of $\Sigma_K$. 
Here is one possible description for divisorial valuations, a straightforward
generalization of the two-dimensional case treated in \cite[Proposition 8.3]{bt0}:

\begin{lemm}
\label{lemm:sigma-struct}
Let $K=k(X)$ be the function field of an algebraic variety of dimension $n\ge 2$. 
Let $\sigma_1,\sigma_2\in \Sigma_K$ be liftable subgroups of rank $n$ such that 
$\I:=\sigma_1\cap \sigma_2$ is topologically cyclic. Then there exists a unique divisorial valuation 
$\nu$ such that $\I=\I^a_{\nu}$. The corresponding decomposition group $\D^a_{\nu}\subset \G^a_K$ is the 
subgroup of elements forming a commuting pair with a topological generator of $\I^a_{\nu}$. 
\end{lemm}

\begin{proof}
Let $\nu_1,\nu_2\in \Val_K$ be the valuations associated to 
$\sigma_1,\sigma_2$ in Proposition~\ref{prop:lift}. By Corollary~\ref{coro:intersect}, there
exists a valuation $\nu\in \Val_K$ such that 
$$
\sigma_j\subset \D^a_{\nu_j}\subset \D^a_{\nu}, \quad \text{ for } \quad  j=1,2.
$$
Let $\I^a_{\nu}$ be the corresponding inertia subgroup, 
the subgroup of elements commuting with all of $\D^a_{\nu}$. 
In particular,  $\I^a_{\nu}$ commutes with 
all elements of $\sigma_1$ and $\sigma_2$. Since $\sigma_1,\sigma_2$ 
are maximal liftable subgroups of $\G^a_K$, we obtain that 
$$
\I^a_{\nu}\subseteq \sigma_1\cap \sigma_2=\I\simeq \Z_{\ell}.   
$$
Note that $\I^a_{\nu}$ cannot be trivial; otherwise, the residue field $\KK_{\nu}$ would 
contain a liftable subgroup of rank $n$, and have transcendence degree $n$,   
by Corollary~\ref{coro:sigma-rank}, which is impossible. 
It follows that $\rk_{\Z_{\ell}}(\I^a_{\nu})=1$ and $\trdeg_k(\KK_{\nu})\le n-1$. 

Now we apply Corollary~\ref{coro:sigma-rank} to  
$$
\bar{\sigma}_j:=\sigma_j/\I^a_{\nu}\subset \G^a_{\KK_{\nu}}, \quad \text{ for }  \quad j=1,2,
$$ 
liftable subgroups  of rank $n-1$. It follows that $\trdeg_k(\KK_{\nu})\ge n-1$, thus 
equal to $n-1$, i.e., $\nu$ is a divisorial valuation.

Conversely, an inertia subgroup $\I^a_{\nu}$ can be embedded into maximal liftable subgroups
$\sigma_1,\sigma_2$ as above, e.g., 
by considering ``flag'' valuation with value group $\Z^n$, with disjoint centers 
supported on the corresponding divisor $D=D_{\nu}\subset X$.  
\end{proof}

The following is useful for the visualization of composite valuations: 

\begin{lemm}
\label{lemm:liftable}
Let $\nu\in \DVal_K$ be a divisorial valuation. 
There is a bijection between 
liftable subgroups $\sigma\in \Sigma_K$ with the property that 
$$
\I^a_{\nu}\subset \sigma \subseteq \D^a_{\nu} 
$$
and liftable subgroups $\sigma_{\nu}\in \Sigma_{\KK_{\nu}}$. 
\end{lemm}

\begin{proof}
We apply \cite[Corollary 8.2]{bt0}:
let $\nu$ be a valuation of $K$ and $\iota_{\nu}\in \I^a_{\nu}$. Let 
$\gamma\in \G^a_K$ be such that $\iota_{\nu}$ and $\gamma$ form a commuting pair. 
Then $\gamma\in \D^a_{\nu}$. 
\end{proof}

\

In summary, under the assumptions of Theorem~\ref{thm:main}, we have obtained:
\begin{itemize}
\item a canonical isomorphism of completions 
$\Psi^*\,:\, \hat{L}^*\stackrel{\sim}{\lra} \hat{K}^*$ induced, by 
Kummer theory, from the isomorphism $\Psi\,:\, \G^a_K\stackrel{\sim}{\lra} \G^a_L$;
\item a bijection on the set of inertia (and decomposition) subgroups of divisorial valuations 
$$
\G^a_K\supset \I^a_{\nu} \stackrel{\Psi}{\longrightarrow} \I^a_{\nu} \subset \G^a_L.
$$
\end{itemize}
Note that $K^*/k^*\subset \hat{K}^* $ determines a canonical
topological generator $\delta_{\nu,K}\in \I^a_{\nu}$, for all $\nu\in \DVal_K$,
by the condition that the restriction takes values in the integers 
$$
\delta_{\nu,K} \,:\, K^*/k^* \ra \Z\subset \Z_{\ell}
$$
i.e., that there exist elements $f\in K^*/k^*$ such that $\delta_{\nu,K}(f)=1$.  
A  topological generator of the procyclic group $\I^a_{\nu}\simeq \Z_{\ell}$ 
is defined up to the action of $\Z_{\ell}^*$.
We conclude that there exist constants 
$$
\varepsilon_{\nu}\in \Z_{\ell}^*, \quad  \nu\in \DVal_K=\DVal_L
$$
such that 
\begin{equation}
\label{eqn:epsilon}
\Psi(\delta_{\nu,K}) =\varepsilon_{\nu} \cdot \delta_{\nu,L},\quad \forall\,  \nu\in \DVal_K. 
\end{equation} 
The main difficulty is to show that there exists a conformally
{\em unique} $\Z_{(\ell)}$-lattice, i.e., a constant $\epsilon\in \Z_{\ell}^*$, unique
modulo $\Z_{(\ell)}^*$, such that
$$
\varepsilon_{\nu}=\epsilon, \quad \forall \nu\in \DVal_K.
$$
A proof of this fact will be carried out in Section~\ref{sect:1-dim}. 

\

Let $\nu$ be a divisorial valuation. Passing to $\ell$-adic 
completions in sequence \eqref{eqn:1}
we obtain an exact sequence
$$
1\ra \hat{\mo}_{\nu}^*\ra \hat{K}^*\stackrel{\nu}{\longrightarrow} \Z_{\ell}\ra 0.
$$
The sequence  \eqref{eqn:2}
gives rise to a surjective homomorphism
$$
\hat{\mo}_{\nu}\ra \hat{\KK}^*_{\nu}. 
$$
Combining these, we obtain
a surjective homomorphism 
\begin{equation}
\label{eqn:residue-field}
\res_{\nu}\,:\, \Ker(\nu)\ra \hat{\KK}^*_{\nu}. 
\end{equation}
This homomorphism has a Galois-theoretic description, 
via duality arising from Kummer theory:  
We have
$$
\I^a_{\nu} \subset \D^a_{\nu}\subset \G^a_K,
$$
and
$$
\hat{\KK}^*_{\nu}=\Hom(\G^a_{\KK_{\nu}},\Z_{\ell}) = \Hom(\D^a_{\nu}/\I^a_{\nu}, \Z_{\ell});
$$
each $\hat{f}\in \Ker(\nu)\subset \hat{K}^* = \Hom(\G^a_K, \Z_{\ell})$ 
gives rise to a well-defined element in $\Hom(\D^a_{\nu}/\I^a_{\nu}, \Z_{\ell})$.

\section{$\ell$-adic analysis: generalities}
\label{sect:ella}

Here we recall the main issues arising in the analysis of $\ell$-adic completions
of functions, divisors, and Picard groups of normal projective models
of function fields $K=k(X)$ (see \cite[Section 11]{bt0} for more details).  

\

We have an exact sequence 
\begin{equation}
\label{eqn:seqq}
0\ra K^*/k^*\stackrel{\dv_X}{\lra} \Div(X)\stackrel{\pic}{\lra} \Pic(X)\ra 0, 
\end{equation}
where $\Div(X)$ is the group of (locally principal) Weil divisors of $X$ and $\Pic(X)$ is the 
Picard group.
We will identify an element $f\in K^*/k^*$ with its image under $\dv_X$.  Let 
$$
\widehat{\Div}(X)
$$ 
be the pro-$\ell$-completion of $\Div(X)$ and  put
$$
\Div(X)_{\ell}:=\Div(X)\otimes_{\Z}\Z_{\ell}\subset \widehat{\Div}(X).
$$
Every element $\hat{f}\in \hat{K}^*$ has a  representation
$$
\hat{f}=(f_n)_{n\in \N} 
\,\, {\rm or }\,\, f=f_0f_1^{\ell}f_2^{\ell^2}\cdots ,
$$
with $f_n\in K^*$. 
We have homomorphisms 
$$
\begin{array}{rcl}
\dv_{X} \,: \,\hat{K}^* & \ra  & \widehat{\Div}(X),\\
\hat{f}   \,\,\,\,               & \mapsto &  
\dv_X(\hat{f}):=\sum_{n\in \N}\ell^n\cdot\dv_X(f_n)=
\sum_{m}\hat{a}_m D_m, \\
\end{array}
$$
where $D_m\subset X$ are irreducible divisors, 
$$
\hat{a}_m=\sum_{n\in \N} a_{nm}\ell^n\in \Z_{\ell}, \quad a_{nm}\in \Z.
$$

Equation~\eqref{eqn:seqq} gives rise to an exact sequence
\begin{equation}
\label{eqn:seqq-times}
0\ra K^*/k^*\otimes \Z_{\ell}\stackrel{\dv_{X}}{\lra} \Div^0(X)_{\ell}\stackrel{\pic_{\ell}}{\lra} 
\Pic^0(X)\{\ell\}\ra 0, 
\end{equation}
where 
$$
\Div^0(X)_{\ell}:=\Div(X)^0\otimes \Z_{\ell}, \quad \text{ and } \quad 
\Pic^0(X)\{\ell\}=\Pic^0(X)\otimes \Z_{\ell}
$$ 
is the $\ell$-primary component of the torsion group $\Pic^0(X)$. 
The assignment
$$
\cT_{\ell}(X):=\lim_{\longleftarrow}{\rm Tor}_1(\Z/\ell^n,\Pic^0(X)\{\ell\}).
$$ 
is functorial:
\begin{equation}
\label{eqn:tl-funct}
Y\ra X \quad \Rightarrow \quad \cT_{\ell}(X)\ra \cT_{\ell}(Y).
\end{equation}
We have $\cT_{\ell}(X)\simeq \Z_{\ell}^{2\mathsf g}$, 
where $\mathsf g$ is the dimension of $\Pic^0(X)$.
Passing to pro-$\ell$-completions in \eqref{eqn:seqq-times} we obtain an exact sequence:
\begin{equation}
\label{eqn:seqq-pro-ell}
0\ra \cT_{\ell}(X)\ra \hat{K}^*\stackrel{\dv_X}{\lra} \widehat{\Div^0}(X) \lra 0, 
\end{equation}
since $\Pic^0(X)$ is an $\ell$-divisible group. 
Note that all groups in this sequence are torsion-free. 
We have a diagram   
\begin{equation}
\label{eqn:dia-need}
\begin{array}{ccccccccccc}
  &     &    0         & \ra & K^*/k^*\otimes\Z_{\ell} &  \stackrel{\dv_{X}}{\lra}   &   \Div^0(X)_{\ell}  &\stackrel{\pic_{\ell}}{\lra}  & \Pic^0(X)\{\ell\}  & \ra &  0\\ 
  &     &              &     &   \downarrow             &                 &   \downarrow      &     &      \downarrow &     &   \\
0 & \ra &  \cT_{\ell}(X)    & \ra & \hat{K}^*  & \stackrel{\dv_{X}}{\lra} & \widehat{\Div^0}(X)& 
\stackrel{\hat{\pic}}{\lra} & 0.  & 
\end{array}
\end{equation}

Every $\nu\in \DVal_K$ gives rise to a
homomorphism
$$
\nu\,:\, \hat{K}^*\ra \Z_{\ell}.
$$
On a normal model $X$, where $\nu=\nu_D$ for some divisor $D\subset X$, 
$\nu(\hat{f})$ is the $\ell$-adic coefficient at $D$ of $\dv(\hat{f})$.

The following lemma generalizes \cite[Lemmas 11.12 and 11.14]{bt0} 
to normal varieties.

\begin{lemm}
\label{lemm:alba}
Let $K$ be a function field over $k$ of transcendence degree $\ge 3$. 
Then there exists a normal projective model 
$X$ of $K$ such that 
for all birational maps $\tilde{X}\ra X$ 
from a normal variety $\tilde{X}$
one has a canonical isomorphism
$$
\cT_{\ell}(X)\ra \cT_{\ell}(\tilde{X}).
$$
In particular, $\cT_{\ell}(X)$ is an invariant of $K$. 
Moreover, we have
\begin{equation}
\label{eqn:ctl}
\cT_{\ell}(X)=\cT_{\ell}(K)=\cap_{\nu\in \DVal_K} \Ker(\hat{\nu}) \subset
\hat{K}^*.
\end{equation}
\end{lemm}

\begin{proof}
For any projective $X$, its Albanese $\Alb(X)$ is the maximal abelian variety 
such that 
\begin{itemize}
\item there exists a morphism $X\ra \Alb(X)$ and
\item $\Alb(X)$ is generated, as an algebraic group, by the image of $X$.
\end{itemize}
This construction is functorial. 
Then there exists an abelian variety $\Alb(K)$ which is maximal for all such models. 
Indeed, the dimension of $\Alb(X)$ is bounded by the genus of a 
flexible curve on any  birational model of $X$. Thus there exists a maximal 
$\Alb(K)$ dominating all $\Alb(X)$ and a class of normal models where 
$\Alb(X)=\Alb(K)$. 
It suffices to observe that $\cT_{\ell}(X)=\cT_{\ell}(\Alb(K))$.

The second claim follows from the fact that every divisorial valuation 
can be realized as a divisor on a normal model $X$ of $K$. 
\end{proof}

\begin{lemm}
\label{lemm:tl}
Let $K=k(X)$ be the function field of a normal
projective variety $X\subset \P^N$ of dimension $\ge 3$. 
For every divisorial valuation $\nu\in \DVal_K$ there is a canonical
homomorphism:
$$
\xi_{\nu,\ell} \,:\, \cT_{\ell}(K)\ra \cT_{\ell}(\KK_{\nu}).
$$
Assume that $\nu$ corresponds to an irreducible normal 
hyperplane section of $X$. 
Then $\xi_{\nu,\ell}$ is an isomorphism.  
\end{lemm}

\begin{proof}
The map is induced from a canonical map of Albanese varieties
(see \cite[Lemma 11.12]{bt0}). It suffices to apply Lefschetz' theorem.
\end{proof}

\begin{lemm}
\label{lemm:lef}
Let $\lambda\,:\, X\ra \P^1$ be a Lefschetz pencil on
a normal variety of dimension $\ge 3$ and $D_t=\lambda^{-1}(t)$. Then: 
\begin{enumerate}
\item For all but finitely many $t\in \P^1$, 
$$
\xi_{D_t,\ell}\,:\, \cT_{\ell}(X) \stackrel{\sim}{\lra} \cT_\ell(D_t),
$$
is an isomorphism.
\item
For any $t\in \P^1$ and any surjection $D_t\ra C_t$ 
onto a smooth projective curve 
we have $\mathsf g(C_t)\le \rk_{\Z_{\ell}}(\cT_{\ell}(X))$.
\end{enumerate}
\end{lemm}

\begin{proof}
Follows from standard facts 
for general hyperplane sections of 
normal varieties (see Lemma~\ref{lemm:tl}).  
\end{proof}

\begin{lemm} 
\label{lemm:pencil}
Let $\pi\, :\, X\to C$ be a surjective map with irreducible fibers.
Assume that $\hat{f}\in \Ker(\hat{\nu})$ and that 
$\res_{\nu}(\hat{f})=1\in \hat{\KK}_{\nu}^*$, for infinitely many $\nu\in \DVal_K$
corresponding to fibers of $\pi$.
Then $\hat{f}$ is induced from $\widehat{k(C)}^*$.
\end{lemm}

\begin{proof} 
Assume that $\hat{f} \mod \ell^n$, for some $n\in \N$, 
contains a summand corresponding to a horizontal divisor $R$.
By Lemma~\ref{lemm:ea1}, 
$R$ intersects all but finitely many fibers $p^m$-transversally. 
In particular, $\dv_X(\hat{f})$ intersects infinitely many fibers nontrivially, contradiction
to the assumption.  
Thus $\dv_X(\hat f)$ is a sum of vertical divisors.

Hence $\hat{f} = \tau + \hat{g} $, where $\hat{g}\in \widehat{k(C)}^*$, and $\tau\in \cT_\ell(K)$. 
The triviality of $\tau $ on fibers $D_c=\pi^{-1}(c)$ implies that
$\tau$ is induced from the image of $X$ in $\Alb(X)/\Alb(D_c)$.
In particular, the triviality on infinitely many fibers implies
that it is induced from the Jacobian $J(C)$ and hence $\hat{f} \in \widehat{k(C)}^*$.
\end{proof}

\begin{nota}
\label{nota:supp}
Let $X$ be a normal projective model of $K$. 
For $\hat{f}\in \hat{K}^*$ with 
$$
\dv_X(\hat{f}) = \sum_{m} \hat{a}_m D_m
$$
we put
$$
\begin{array}{rcccl}
\supp_K(\hat{f})& := \{ &  \nu\in \DVal_K & | &  \hat{f} \,\,\, 
{\rm nontrivial\,\,\, on} 
\,\,\,\I^a_{\nu}\,\,\};\\
\supp_X(\hat{f})& := \{ &  D_m \subset X           & | &   \hat{a}_m\neq 0\,\,\};\\
\fibr(\hat{f}) &  := \{ &  \nu\in \DVal_K & | & \hat{f}\in \Ker(\hat{\nu})\,\text{ and }\,\, 
\res_{\nu}(\hat{f})=1\in \hat{\KK}_{\nu}\,\,\},
\end{array}
$$
where $\res_{\nu}$ is the projection from Equation~\eqref{eqn:residue-field}.
Note that the {\em finiteness} of 
$\supp_X(\hat{f})$ does not depend on the choice of the normal model $X$. 
Put
$$
\supp'_K(\hat{f}):=\fibr(\hat{f}) \cup \supp_K(\hat{f}).
$$
If $X$ is a normal model of $K$ write
$$
\supp'_X(\hat{f})\subset \supp'_K(\hat{f})
$$
for the subset of divisorial valuations realized by divisors on $X$. We have
$$
\supp'_K(\hat{f}) = \cup_X \,\,\supp'_X(\hat{f}).
$$
\end{nota}

\begin{defn}
A $K$-divisor is a function
$$
\DVal_K \ra \Z_{\ell}.
$$ 
Each $\hat{f}\in \hat{K}^*$ defines a $K$-divisor by
$$
{\rm div}_K(\hat{f}) \colon \quad \nu \mapsto [\delta_{\nu,K}, \hat{f}].
$$
\end{defn}

The different notions of support for elements in 
$\hat{K}^*$ introduced in Notation~\ref{nota:supp}
extend naturally to $K$-divisors. 
The divisor of $\hat f$ on a normal model $X$ of $K$ coincides
with the restriction of ${\rm div}_K(\hat{f})$ to the set of divisorial
valuations of $K$ which are realized by divisors on $X$.
In particular, it has finite support on $X$ modulo $\ell^n$, for any $n\in \N$.
(This fails for general $K$-divisors.)

\

Let $E\subset K$ be a one-dimensional subfield and $\pi_E\,:\, X\ra C$
the corresponding surjective map with irreducible generic fiber. 
For all nontrivial $\hat{f}_1,\hat{f}_2\in \hat{E}^*$,  
we have 
$$
\supp'_K(\hat{f}_1)=\supp'_K(\hat{f}_2).
$$
This gives a well-defined invariant of $\hat{E}^*$.  
We have a decomposition 
\begin{equation}
\label{eqn:decomp} 
\supp'_K(\hat{E}^*)= \sqcup_{c\in C} \,\,\supp'_{K,c} (\hat{E}^*),
\end{equation} 
where $\supp'_{K,c}(\hat{E}^*)$ 
are minimal nonempty subsets of the form 
$$
\supp_K(\hat{f}_1)\cap \supp_K(\hat{f}_2)
$$
contained in 
$\supp'_K(\hat{E}^∗)$; 
these correspond to sets of 
irreducible divisors supported in $\pi_E^{-1}(c)$,  for $c\in C (k)$. 
Note that $\supp'_K(\hat{E}^*)$ depends only on the 
normal closure of $E$ in $K$. On the other hand, 
the decomposition \eqref{eqn:decomp} is 
preserved only under purely inseparable extensions of $E$. 
We formalize this discussion 
in the following definition.

\begin{defn}
A formal projection is a triple 
$$
\pi_{\hat{E}}=(C,\{ R_c\}_{c\in C}, Q),
$$ 
where
$C$ is an infinite set, $\{ R_c\}_{c\in C}$ is a set of $K$-divisors, and  
$Q\subset\hat K^*$ a subgroup of $\Z_{\ell}$-rank at least two
satisfying the following properties:
\begin{enumerate}
\item for all $\hat{f}_1,\hat{f}_2\in Q$ one has 
$\supp_K'(\hat{f}_1 )= \supp_K'(\hat{f}_2)$;
\item $\supp_K(R_{c_1}) \cap  \supp_K(R_{c_2})=\emptyset$, for all pairs of distinct $c_1,c_2\in C$;
\item for all nontrivial $\hat{f}\in Q$ one has
$$
{\rm div}_K(\hat f)=\sum_{c\in C} a_c R_c, \quad a_c\in \Z_\ell,
$$
and
$$
\cup_{c\in C} \supp_K(R_c) = \supp_K'(\hat{f});
$$ 
\item for all $c_1,c_2\in C$ there exists an $m\in \N$ such that
$$
m(R_{c_1}-R_{c_2})= {\rm div}_K(\hat f),
$$ 
for some $\hat f\in Q$.
\end{enumerate}
\end{defn}

\begin{exam}
A one-dimensional subfield $E=k(C)\subset K$ defines a 
formal projection $\pi_{\hat{E}} = (C, \{R_c\}_{c\in C}, Q)$, 
with $C$ the set of $k$-points of the image of $\pi_E$, $R_c$ the intrinsic 
$K$-divisors over $c\in C$, and $Q=\hat{E}^*$.

Note that for normally closed subfields $E\subset K$, 
the corresponding subgroup $Q$ is maximal, for subgroups of $\hat{K}^*$
appearing in formal projections.  
\end{exam}

\begin{lemm}
\label{lemm:image-rc}
The formal divisor $\dv_X(R_c)$ is finite $\mod \ell^n$ for any model $X$.
\end{lemm}

\begin{proof} 
The support of  $\Psi^*(\hat f) \mod \ell^n$ is finite for all $n\in \N$.
Now observe that the $K$-divisors $R_c$ have disjoint support in $\supp_K'(Q)$, thus have
no components in common. 
\end{proof}

\section {One-dimensional subfields}
\label{sect:1-dim}

We recall the setup of Theorem~\ref{thm:main}:
$$
\Psi \,:\, \G^a_K\ra \G^a_L.
$$
Our goal here is to show: 

\centerline{
\xymatrix{
\hat{L}^*  \ar[r]^{\Psi^*}  & \hat{K}^* \\
  L^*/l^*  \ar[r]   \ar[u]  &  (K^*/k^*)^{\epsilon} \ar[u]
}
}

\

\noindent
for some constant $\epsilon$.
We know that $g\in K^*/k^*\otimes \Z_{\ell}$ have finite support $\supp_X(g)$, on every 
normal model $X$ of $K$.  In the second half of this
section we will prove:

\begin{prop}[Finiteness of support]
\label{prop:finite-support}
For all $f\in L^*/l^*$ and all normal models $X$ of $K$ 
the support $\supp_X(\Psi^*(f))$ is finite. 
\end{prop}

Assuming this, we will prove:

\begin{prop}[Image of $\Psi^*$]
\label{prop:33}
For all $f\in L^*/l^*$ there exist a function 
$g\in K^*/k^*$ and constants $N\in \N$, $\alpha\in \Z_{\ell}$ such that 
\begin{equation}
\label{eqn:indu}
\Psi^*(f)^N=g^{\alpha}.
\end{equation}
Moreover, there exists a constant $\epsilon\in \Z_{\ell}^*$ such that 
$$
\Psi^*(l(f)^*/l^* \otimes \Z_{(\ell)}) 
\subseteq \left(k(g)^*/k^* \otimes \Z_{(\ell)}\right)^\epsilon.
$$
\end{prop}

Considerations in Section~\ref{sect:val} imply that 
under the assumptions of Theorem~\ref{thm:main}
we have a canonical commutative diagram, for every $\nu\in\DVal_K$: 

\

\centerline{
\xymatrix{
0 \ar[r] & \cT_{\ell}(L) \ar[d]_{\Psi^*} \ar[r] & \Ker(\nu) \ar@{=}[d]_{\Psi_{\nu}^*} \ar[r] &
{\hat{\LL}}^*_{\nu} \ar[d]^{{\Psi^*_{\nu}}} & \supset &  \LL^*_{\nu}\otimes \Z_{(\ell)} \ar[d]_{\Psi^*_{\nu}}  \\
0 \ar[r] & \cT_{\ell}(K) \ar[r]        & \Ker(\nu)  \ar[r] & 
\hat{\KK}_{\nu}^*  & \supset &  \left(\KK^*_{\nu}\otimes \Z_{(\ell)}\right)^\epsilon, 
}
}

\noindent
for some constant $\epsilon\in \Z_{\ell}^*$, depending on $\nu$.  
By \cite[Proposition 12.10]{bt0}, the left vertical map is
a canonical isomorphism. 
In both proofs (Finiteness of support and Image of $\Psi^*$)
we will apply the inductive
assumption~\eqref{eqn:induct} to residue fields of 
appropriate divisorial valuations.

\

\begin{proof}[Proof of Proposition~\ref{prop:33}]
Let $X$ be a normal projective model of $K$ and put $\hat{f}:=\Psi^*(f)$. 
By Proposition~\ref{prop:finite-support}, we may assume
that $\supp_X(\hat{f})$ is finite, i.e., 
$$
{\rm div}(\hat{f}) = \sum_{j\in J} d_jD_j,
$$
where $J$ is a finite set, $d_j\in \Z_{\ell}$ and $D_i$ are irreducible divisors on $X$. 
Then there exists an $N\in \N$ such that $\hat{f}^N\in \Div^0(X)_{\ell}\subset 
\widehat{\Div^0}(X)$.   
By \eqref{eqn:dia-need}, we have 
$$
\hat{f}^N= t_{\hat{f}}\cdot \prod_{i\in I} g_i^{a_i},
$$ 
with $I$ a finite set, $a_i\in \Z_{\ell}$
linearly independent over $\Z_{(\ell)}$, $g_i\in K^*/k^*$ multiplicatively independent,  
and $t_{\hat{f}}\in \cT_{\ell}(K)$.

The projective model $X$ contains a hyperplane section $D\subset X$
such that 
$$
\cT_{\ell}(K)=\cT_{\ell}(X)=\cT_{\ell}(D),
$$
under the natural restriction isomorphism $\xi_{D,\ell}$ from Lemma~\ref{lemm:lef},
and the restrictions of $g_i$ to $D$ are multiplicatively independent in 
$k(D)^*/k^*=\KK_{\nu}^*/k^*$, where $\nu=\nu_D$. 

By the construction and the inductive assumption,  we have
$\res_{\nu}(\hat{f}^N)= g_{\nu}^{b_{\nu}}$, where $b_{\nu}\in \Z_{\ell}$, $g_{\nu}\in \KK^*_{\nu}$:
$$
\res_{\nu}(\hat{f}^N)
=
\res_{\nu}(t_{\hat{f}})\cdot \prod_{i\in I} \res_{\nu}(g_i)^{a_i}
=g_{\nu}^{b_{\nu}}.
$$
In particular, $\res_{\nu}(t_{\hat{f}})=1$ and hence $t_{\hat{f}}=1$. 
Since $\res_{\nu}(g_i)\in \KK_{\nu}^*$ are independent, it follows that $\# I=1$ and 
$$
\hat{f}^N= g^{a}, \quad g\in K^*/k^*, \quad a\in \Z_{\ell}.  
$$
This proves the first claim.

The function $g\in K^*/k^*$ defines a map $\pi\,:\, X\ra C$
from some normal model of $K$  onto a curve,
with generically irreducible fibers. For each $h\in l(f)^*/l^*$, consider 
$\dv_X(\Psi^*(h))\subset \widehat{\Div^0}(X)$.
Then divisors in $\dv_X(\Psi^*(h))$ are $\pi$-vertical. 
Indeed, the restriction of $g$ to a $\pi$-horizontal component $D$ would 
be defined and nontrivial. On the other hand, the restriction of $f$ to $D$ is either not defined
or trivial, contradiction. 
By Lemma~\ref{lemm:pencil}, $\Psi^*(h) \in \widehat{k(C)}^*= \widehat{k(g)}^*$. 

Let $\nu=\nu_D$ be a divisorial valuation such that 
$f$ is defined and nontrivial on $D$. Then 
$$
f\in \LL^*_{\nu}/l^* \text{ and } g \in \KK^*_{\nu}/k^*,
$$
and 
$$
\hat{\LL}^*_{\nu} \supset \widehat{l(f)}^* \stackrel{\Psi^*_{\nu}}{\lra}  
\widehat{k(g)}^*\subset \hat{\KK}_{\nu}^*.
$$
By the inductive assumption, this implies that there exists a constant $\epsilon\in \Z_{\ell}^*$ 
such that 
$$
\Psi^*_{\nu}(l(f)^*/l^* \otimes \Z_{(\ell)}) \subseteq \left(k(g)^*/k^*\otimes \Z_{(\ell)} \right)^{\epsilon},
$$
(see, e.g., \cite[Proposition 13.1]{bt0}).
\end{proof}

We now prove Proposition~\ref{prop:finite-support}. 
Fix a normal projective model $Y$ of $L$. The subfield $F=l(f)$ determines
a surjective map $\pi_F\,:\, Y\ra C$  with irreducible generic fibers. 
For each $c\in C$ we have an intrinsically defined formal sum
\begin{equation}
\label{eqn:rc}
R_c=\sum_{\nu\in \DVal_{L,c}} a_{c,\nu} R_{c,\nu},  \quad a_{c,\nu}\in \N,  
\end{equation}
where $\DVal_{L,c}\subset \DVal_{L}=\DVal_K$ is the subset of divisorial valuations 
supported in the fiber over $c$, $R_{c,\nu}$ is a divisor on some model $\tilde{Y}\ra Y$ 
realizing $\nu$, and $a_{c,\nu}$ are local degrees. Note that $R_c$ do not depend 
on the model $Y$, and that $R_{c_1}$ and $R_{c_2}$ have no
common components, for $c_1\neq c_2$. Furthermore, the sets $\DVal_{L,c}$ have an intrinsic Galois-theoretic
characterization in terms of $\hat{F}^*$: these are
minimal nonempty subsets of the form
$$
\supp_K(\hat{f}_1)\cap \supp_K(\hat{f}_2), \quad f_1,f_2\in \hat{F}^*,
$$
contained in $\supp'_K(\hat{F}^*)$.

For each model $\tilde{Y}\ra Y$ we have a map
$$
R_c\mapsto R_{\tilde{Y},c}:=\sum_{\nu \,:\, D_{\nu}\in \Div(\tilde{Y})} a_{c,\nu} R_{c,\nu},
$$
the fiber over $c$. The divisor of a function $f\in F^*/l^*$ on this model can be written 
as a finite sum
$$
\dv_{\tilde{Y}}(f) = \sum n_c R_{\tilde{Y},c},\quad n_c\in \N.
$$

Given $\{ \delta_{\nu,L}\}$, each $\hat{f}\in \hat{L}^*$ defines a $\Z_{\ell}$-valued function 
on $\DVal_L$ by the Kummer-pairing from Theorem~\ref{thm:ga}
\begin{equation}
\label{eqn:hatf}
\begin{array}{ccc}
\DVal_L & \ra      & \Z_{\ell} \\
\nu     & \mapsto  &  [\delta_{\nu,L}, \hat{f}].
\end{array}
\end{equation}
Similarly, each $R_c$ defines a function on $\DVal_L$ by setting 
$$
\nu \mapsto \delta_{\nu,L} \cdot R_c = \delta_{\nu,L}(t), 
$$
where $t$ is a local parameter along $c$ if $\nu$ is supported over $c$, and 
$\nu\mapsto 0$, otherwise. 

For $\hat{f}\in\hat{F}^*\subset \hat{L}^*$ 
write
$$
\dv_C(\hat{f}) =\sum_{c\in C} b_{\hat{f},c} c, \quad b_{f,c}\in \Z_{\ell},
$$
with decreasing coefficients $b_{\hat{f},c}$. 
Then \eqref{eqn:hatf} is given by
$$
\nu\mapsto b_{\hat{f},c} a_{\nu,c}.  
$$

\

We face the following difficulty: we don't know the image $\Psi^*(F^*/l^*)$ in $\hat{K}^*$,
and in particular, we don't know that $\Psi^*(R_c)$, resp. $\Psi^*(R_{\tilde{Y},c})$, 
as functionals on $\DVal_K$, correspond to fibers of any fibration on a model $X$ of $K$. 
However, we know the ``action'' of $\Psi^*$ on the coefficients in Equation~\eqref{eqn:rc}:
$$
a_{c,\nu} \mapsto \varepsilon_{\nu}^{-1} a_{c,\nu}. 
$$

\begin{lemm}
\label{lemm:rc}
Either there is a nonconstant $f\in F^*/l^*$ such that $\supp_X(\Psi^*(f))$ is finite 
or there is at most one $c\in C$, where $C$ corresponds to $F$, 
such that $\Psi^*(R_c)$ 
has finite support on every model $X$ of $K$. 
\end{lemm}

\begin{proof}
Let $c_1,c_2\in C$ be distinct points
such that 
$$
\supp_{X}(\Psi^*(R_c)) \cup \supp_{X}(\Psi^*(R_{c'}))
$$
is finite. Then there is a function $f$ with 
divisor supported in this set, thus finite $\supp_{X}(\Psi^*(f))$. 
\end{proof}

\begin{proof}[Proof of Proposition~\ref{prop:finite-support}]
By contradiction. Assume that $\supp_X(\Psi^*(f))$ is infinite. 
An argument as in the proof of Proposition~\ref{prop:33} shows that the same holds for
every $h\in l(f)^*/l^*$. 

Fix a Lefschetz pencil $\lambda\,:\, X\ra \P^1$ such that 
for almost all fibers $D_t$ of $\lambda$ we have a well-defined
$$
\res_{\nu}  \,:\, l(f)^*/l^* \ra \LL^*_{\nu_t} \stackrel{\Psi^*}{\lra} \widehat{\KK}^*_{\nu_t},
$$
where $\nu_t$ is the divisorial valuation corresponding to $D_t$. 
By the inductive assumption, there exist one-dimensional closed subfields
$E_t=k(C_t)\subset k(D_t)=\KK_{\nu_t}$ such that 
$$
\Psi^*(\res_{\nu_t}(l(f)^*/l^*)\otimes \Z_{(\ell)})
\subseteq \left(E_t^*\otimes \Z_{(\ell)}\right)^{\epsilon_t}, \quad \epsilon_t\in\Z_{\ell}^*.
$$
We have an induced surjective map
$$
\pi_t\,:\, D_t\ra C_t
$$
as in Lemma~\ref{lemm:subf}. Passing to a finite purely-inseparable cover of $C_t$
we may assume that $\pi_t$ is separable (this effects the constant $\epsilon$ by multiplication by a
power of $p$ which is in $\Z_{\ell}^*$). We identify the sets $C(k)$ and $C_t(k)$, set-theoretically. 

Fix a family of flexible curves $\{T_t\}$ uniformly on all but finitely many $D_t$
as in Lemma~\ref{lemm:genus-generic} and let $m$ be the bound on the genus 
of these curves obtained in this Lemma. 
Put $N:=m+4$ and choose $c_{1},\ldots, c_N \in C_t(k)=C(k)$
such that $\supp_{X}(R_{c_j})$ is infinite for all $j$, this is possible by Lemma~\ref{lemm:rc}.

For each $c_j$ express
the fiber over $c_j$ as
$$
R_{c_j}:=\sum_{e=0}^{\infty} \ell^e R_{c_j,e}, \quad
R_{c_j,e}:=\sum_{i\in I_{e,j}} \epsilon_{i,e,j} R_{i,e,j},
$$
where $I_{e,j}$ are finite, and $R_{i,e,j}$ irreducible 
divisors over $c_j$, and $\epsilon_{i,e,j}\in \Z_{\ell}^*$ (see Lemma~\ref{lemm:image-rc}). 
Let $S_{c_j,e}=\cup R_{i,e,j}$ be the support of $R_{c_j,e}$. 
Note that $T_t$ intersect all $S_{c_j,e}$ and write  $d_{j,e}:=\deg(S_{c_j,e}\cdot T_t)$
for the degree of the intersection. 
Choose $M$ such that for all $j=1,\ldots, N$ one has
\begin{equation}
\label{eqn:degree}
d_{j,0}< \sum_{e=1}^M d_{j,e},
\end{equation}
this is possible since the number of components over all $c_j$ is infinite. 
Using Lemma~\ref{lemm:ea1} choose $t$ so that the intersections
$$
R_{i,e,j,t}:= D_t \cdot R_{i,e,j}
$$
are $p$-irreducible and pairwise distinct, 
this holds for all but finitely many $t$. 
Choose a flexible curve $T_t\subset D_t$ such that 
\begin{itemize}
\item $T_t$ does not pass through the points of indeterminacy of 
$\pi_t\,:\, D_t\ra C_t$;
\item $T_t$ is not contained in any of the $R_{i,e,j,t}$;
\item $T_t$ does not pass through pairwise intersections of these divisors. 
\end{itemize}

Consider the restriction
$$
\pi_t\,:\, T_t\ra C_t.
$$
By the choice of $T_t$, the number of nonramified points over each $c_j$ is at most
$d_{j,0}$. On the other hand, the ramification index over $c_j$ is at least 
$\ell \cdot \sum_{e=1}^m d_{j,e}$. 
By the choice \eqref{eqn:degree}, combined with Hurwitz formula in Lemma~\ref{lemm:hurww}, 
we obtain that $\mathsf g(T_t)> m$, contradicting the universal bound.  
\end{proof}

\begin{prop}
\label{prop:ccc}
There exists a constant $\epsilon\in \Z_{\ell}^*$ such that
\begin{equation}
\label{eqn:eps-need}
\Psi^*(L^*/l^* \otimes \Z_{(\ell)})= (K^*/k^* \otimes \Z_{(\ell)})^{\epsilon}.
\end{equation}
\end{prop}

\begin{proof}
By Proposition~\ref{prop:33}, for each 
one-dimensional subfield $F=l(f)\subset L$ 
there exists a one-dimensional subfield $E=k(g)$ and a constant
$\epsilon_F \in \Z_{\ell}^*$ such that 
$$
\Psi^*(F^*/l^* \otimes \Z_{(\ell)})\subseteq (E^*/k^* \otimes \Z_{(\ell)})^{\epsilon_F}.
$$
Moreover, for $f_1,f_2\in L^*/l^*$ we have $\Psi^*(f_j)^{N_j}=g_j^{\epsilon_j}$, 
for some $N_j\in \N$ and  $\epsilon_j\in \Z_{\ell}$. 
It follows that the symbol $(f_1,f_2)$ is infinitely divisible in ${\rm K}^2_M(L)$ if and only if 
$(g_1,g_2)$ is infinitely divisible in  ${\rm K}^2_M(K)$. 
Thus $f_1,f_2$ are algebraically dependent if and only if $g_1,g_2$ are algebraically dependent, 
see Lemma~\ref{lemm:milnor}. 
In particular, if $f_1,f_2$ are not powers of the same element in $L^*$ 
the same holds for $g_1,g_2$, i.e., the  divisors of $g_1,g_2$, on any model $X$ of $K$, are not proportional.
We have
$$
\Psi^*(f_1)^{N_1} =g_1^{\alpha_1}, \quad \Psi^*(f_2)^{N_2} =g_2^{\alpha_2}, \quad \text{ and } \quad
\Psi^*(f_1f_2)^{N_{12}} =  g_{12}^{\alpha_{12}}.
$$ 
We need to show that $\alpha_1,\alpha_2, \alpha_{12}$ span a 1-dimensional lattice, modulo $\Z_{(\ell)}$. 
We have
$$
g_{12}= g_1^{\alpha_1/\alpha_{12} N_1} g_2^{\alpha_2/\alpha_{12} N_2}, \quad \text{ for some }  N_{1},{N_2}\in \Z_{(\ell)}. 
$$
In particular, for any divisorial valuation $\nu$ in the support of $g_{12}$ 
the integral coefficient $b_{12}(\nu)=[\delta_{\nu,K},g_{12}]\in \Z$ equals
$$
b_1(\nu)\alpha_1/\alpha_{12} N_1 + b_2(\nu) \alpha_2/\alpha_{12} N_2, \quad \text{ for some } \quad b_1(\nu),b_2(\nu)\in \Z.
$$
Since the divisors of $g_1,g_2$ are not proportional
the rank of the corresponding system of equations, as we vary over $\nu$,
is at at least $2$.
Hence  both $\alpha_1/\alpha_{12} N_1, \alpha_2/\alpha_{12} N_2$ are rational, as claimed.  
\end{proof}

\

\section{Proof}
\label{sect:proof}

In this section we prove our main theorem.

\

{\em Step 1.} We have a nondegenerate pairing 
$$
\G_K^a\times \hat{K}^*\ra \Z_{\ell}(1).   
$$
This implies a canonical isomorphism
$$
\Psi^*\,:\, \hat{L}^*\ra \hat{K}^*.
$$ 

\

{\em Step 2.}
By assumption, $\Psi\,:\, \G^a_K\ra \G^a_L$ 
is bijective on the set of liftable subgroups, in 
particular, it maps liftable subgroups $\sigma\in \Sigma_K$ 
to a liftable subgroups of the same rank. 
In Section~\ref{sect:val} we identify intrinsically 
the inertia and decomposition groups of divisorial valuations:
$$
\I^a_\nu\subset\D^a_\nu\subset \G^a_K:
$$ 
every liftable subgroup
$\sigma\in \Sigma_K$ contains an inertia element
of a divisorial valuation 
(which is also contained in at least one other $\sigma'\in \Sigma_K$). The
corresponding decomposition group is the ``centralizer'' of the 
(topologically) cyclic inertia group 
(the set of all elements which ``commute'' with inertia). 
This identifies $\DVal_K=\DVal_L$. 

\

By \cite[Section 17, Step 7 and 8]{bt0}, an isomorphism
$$
\Psi^* \,:\, \G^a_K\ra \G^a_L
$$
of abelianized Galois group of 
function fields $K=k(X)$ and $L=l(Y)$ 
of surfaces over algebraic closures of 
finite fields of characteristic $\neq \ell$
implies the existence of a constant 
$\epsilon\in \Z_{\ell}^*$ and a canonical isomorphism
$$
L^*/l^*\otimes \Z_{\ell}\supset \cup_{n\in \N} (L^*/l^*)^{1/p^n}\simeq  
\cup_{n\in \N}(K^*/k^*)^{\epsilon/p^n} \subset K^*/k^*\otimes \Z_{\ell}.
$$
By the induction hypothesis, we may assume that this holds in dimension $\le n-1$: 
Once we have identified decomposition and inertia subgroups of divisorial valuations, 
we have, for each $\nu\in \DVal_K$,  an intrinsically defined sublattice
$$
\Psi^*(\LL^*_{\nu}/l^*)=(\KK_{\nu}^*/k^*)^{\epsilon}\subset \hat{\KK}_{\nu}^*
$$ 
of elements of the form $f^{\epsilon}$, with $f\in \KK_{\nu}^*/k^*$ and $\epsilon\in \Z_{\ell}^*$
in the completion of the residue field. 
Using Proposition~\ref{prop:ccc}, we prove that the same holds for the image of $L^*/l^*$ in 
$\hat{K}^*$.

\

Thus we obtained an isomorphism
$$
\epsilon^{-1}\cdot \Psi^* \,: \,  L^*/l^*\otimes \Z_{(\ell)} \ra K^*/k^*  \otimes \Z_{(\ell)}
$$
which maps multiplicative groups of one-dimensional
subfields of $L$ into multiplicative groups of 
one-dimensional subfields of $K$. 
The same holds for multiplicative groups of 
subfields of transcendence degree two (see Proposition~\ref{prop:ccc} and its proof; 
Lemma~\ref{lemm:milnor} allows to characterize algebraically independent elements).
To conclude the proof it suffices to apply the inductive hypothesis, the case of surfaces: 
in \cite{bt0} we showed that the additive structure is canonically encoded in these data.

\

\bibliographystyle{smfplain}
\bibliography{recon}

\end{document}